\documentclass[12pt]{article}
\usepackage{amsfonts,amsmath,amsthm, amssymb}
\usepackage{latexsym, euscript, epic, eepic}
\usepackage{lineno}

\newtheorem{cro}{Corollary}[section]
\newtheorem{defn}{Definition}[section]
\newtheorem{prop}{Proposition}[section]
\newtheorem{thm}{Theorem}[section]
\newtheorem{lem}{Lemma}[section]

\begin{document}

\title{The Bowen's topological entropy of the Cartesian product sets
 \footnotetext {* Corresponding author}
  }
\author{Xiaoyao Zhou $^{\rm a}$,  Ercai Chen$^{\rm a, b *}$ \\
 \small \it $^{\rm a}$ School of Mathematical Sciences and Institute of Mathematics, Nanjing Normal University,\\
  \small \it Nanjing 210023, Jiangsu, P.R.China\\
      \small \it $^{\rm b}$ Center of Nonlinear Science, Nanjing University,\\
       \small \it  Nanjing 210093, Jiangsu, P.R.China\\
        \small \it  e-mail: zhouxiaoyaodeyouxian@126.com,\\
         \small \it ecchen@njnu.edu.cn\\
}
\date{}
\maketitle

\begin{center}
\begin{minipage}{120mm}
{\small {\bf Abstract.} This article is devoted to showing the
product theorem for Bowen's topological entropy. }
\end{minipage}
\end{center}

\vskip0.5cm
\section{Introduction and Preliminaries}
The purpose of this article is to study the topological entropies of
product spaces. The product theorem for topological entropy of the
dynamical systems was first investigated by  Adler,  Konheim and
McAndrew \cite{AdlKonMca} and Goodywn  \cite{Goo}. One can see
\cite{Wal} for the product theorem for topological entropy of two
compact subsets. Bowen \cite{Bow} introduced the notion of
topological entropy for non-compact sets. A question arises
naturally whether the product theorem for Bowen's topological
entropy still holds. The reader is also referred to
\cite{BesMor,Fed,How,Kel,Mat,Mae}and references therein for the
investigation of dimension of product spaces.

 Throughout this article, a {\it
topological dynamical system} $(X,d,T)$ means a compact metric space
$(X,d)$ together with a continuous self-map $T:X\to X.$ Let $M(X)$
and $M(X,T)$ denote
 the sets of all Borel probability measures and
$T$-invariant Borel probability measures, respectively. For
$n\in\mathbb{N},$ the $n$-th {\it Bowen metric} $d_n$ on $X$ is
defined by

\begin{linenomath*}
 \begin{eqnarray*}
  d_n(x,y)=\max\{d(T^kx,T^ky):k=0,1,\cdots,n-1\}.
   \end{eqnarray*}
    \end{linenomath*}
For every $\epsilon>0,$ denote by
$B_n(x,\epsilon),\overline{B}_n(x,\epsilon)$ the open and closed
balls of radius $\epsilon$ order $n$ in the metric $d_n$ around $x,$
i.e.,

\begin{linenomath*}
 \begin{eqnarray*}
  B_n(x,\epsilon)=\{y\in X: d_n(x,y)<\epsilon\} {\rm~and~}
   \overline{B}_n(x,\epsilon)=\{y\in X: d_n(x,y)\leq\epsilon\}.
    \end{eqnarray*}
     \end{linenomath*}
Recently, given $\mu\in M(X),$ Feng and Huang \cite{FenHua} defined
the {\it measure-theoretical lower and upper entropies} of $\mu$
 respectively   by the idea analogous to Brin and Katok
\cite{BriKat} as follows.

\begin{defn}
Let $\underline{h}_\mu(T)=\int \underline{h}_\mu(T,x)d\mu(x)$ and
$\overline{h}_\mu(T)=\int \overline{h}_\mu(T,x)d\mu(x),$ where

\begin{linenomath*}
 \begin{eqnarray*}
  \underline{h}_\mu(T,x)=\lim\limits_{\epsilon\to0}\liminf\limits_{n\to\infty}-\frac{1}{n}\log\mu(B_n(x,\epsilon)),\\
   \overline{h}_\mu(T,x)=\lim\limits_{\epsilon\to0}\limsup\limits_{n\to\infty}-\frac{1}{n}\log\mu(B_n(x,\epsilon)).
    \end{eqnarray*}
     \end{linenomath*}
      \end{defn}
Brin and Katok \cite{BriKat} proved that for any $\mu\in M(X,T),
\overline{h}_\mu(T,x)=\underline{h}_\mu(T,x)$ for $\mu$-a.e. $x\in
X,$ and $\int \underline{h}_\mu(T,x) d\mu(x)=h_\mu(T).$ This implies
that for any $\mu\in M(X,T),
\underline{h}_\mu(T)=\overline{h}_\mu(T)=h_\mu(T).$

A set in a metric space is said to be {\it analytic} if it is a
continuous image of the set $\EuScript{N}$ of infinite sequences of
natural numbers (with its product topology).  It is well known that
in a Polish space, the analytic subsets are closed under countable
unions and intersections, and any Borel set is analytic (see
\cite{Fed}).


\section{Definitions of Topological Entropies and Main Theorem}
In this section, we recall three definitions of topological
entropies of subsets in a topological dynamical system: Bowen's
topological entropy, packing topological entropy and upper capacity
topological entropy. Since they are analogous to the definitions of
dimensions, they are called dimensional entropies.
\subsection{Bowen's topological entropy}

Bowen's topological entropy was first introduced  in \cite{Bow}.
Here we use an alternative way to define Bowen's topological entropy
for convenience. See \cite{Pes} for details.

Suppose $(X,d,T)$ is a topological dynamical system. Given $Z\subset
X, s\geq0, N\in\mathbb{N}$ and $\epsilon>0,$ define

\begin{linenomath*}
 \begin{eqnarray*}
  \EuScript{M}_{N,\epsilon}^s(Z)=\inf\sum\limits_i\exp(-sn_i),
   \end{eqnarray*}
    \end{linenomath*}
where the infimum ranges over all finite or countable families
$\{B_{n_i}(x_i,\epsilon)\}$ such that $x_i\in X, n_i\geq N$ and
$\bigcup_iB_{n_i}(x_i,\epsilon)\supset Z.$ Since
$\EuScript{M}_{N,\epsilon}^s(Z)$ does not decrease as $N$ increases
and $\epsilon$ decreases, the following two limits exist:

\begin{linenomath*}
 \begin{eqnarray*}
  \EuScript{M}_\epsilon^s(Z)=\lim\limits_{N\to\infty}\EuScript{M}_{N,\epsilon}^s(Z),~~~~\EuScript{M}^s(Z)=\lim\limits_{\epsilon\to0}\EuScript{M}_\epsilon^s(Z).
   \end{eqnarray*}
    \end{linenomath*}
The Bowen's topological entropy $h^B(Z)$ is defined as a critical
value of the parameter $s,$ where $\EuScript{M}^s(Z)$ jumps from
$\infty$ to 0, i.e.,

\begin{linenomath*}
 \begin{eqnarray*}
  \EuScript{M}^s(Z)=
   \left\{\begin{aligned}
0, ~~~&s>h^B(Z),\\
\infty,  ~~~&s<h^B(Z).
     \end{aligned}\right.
      \end{eqnarray*}
       \end{linenomath*}
\subsection{Packing topological entropy}
Packing topological entropy was defined by Feng and Huang
\cite{FenHua} in a way which resembles the packing dimension.
Nowadays, the packing topological entropy is widely believed as
important as the Bowen's topological entropy and  an understanding
of both the Bowen's topological entropy and the packing topological
entropy of a set provides the basis for a substantially better
understanding of the underlying geometry and dynamical behavior of
the set.

Given $Z\subset X, s\geq0, N\in\mathbb{N}$ and $\epsilon>0,$ define

\begin{linenomath*}
 \begin{eqnarray*}
  P^s_{N,\epsilon}(Z)=\sup\sum\limits_i\exp(-sn_i),
   \end{eqnarray*}
    \end{linenomath*}
where the supremum runs over all finite or countable pairwise
disjoint families $\{\overline{B}_{n_i}(x_i,\epsilon)\}$ such that
$x_i\in Z, n_i\geq N$ for all $i.$ Since $P^s_{N,\epsilon}(Z)$ does
not decrease as $N,\epsilon$ decrease, the following limit exists:

\begin{linenomath*}
 \begin{eqnarray*}
  P^s_{\epsilon}(Z)=\lim\limits_{N\to\infty}P^s_{N,\epsilon}(Z).
   \end{eqnarray*}
    \end{linenomath*}
Define

\begin{linenomath*}
 \begin{eqnarray*}
  \EuScript{P}^s_{\epsilon}(Z)=\inf\left\{\sum\limits_{i=1}^\infty
   P^s_{\epsilon}(Z_i):\bigcup\limits_{i=1}^\infty Z_i\supset
    Z\right\}.
    \end{eqnarray*}
     \end{linenomath*}
It is obvious that for $Z\subset\bigcup\limits_{i=1}^\infty Z_i,
\EuScript{P}^s_{\epsilon}(Z)\leq\sum\limits_{i=1}^\infty\EuScript{P}^s_{\epsilon}(Z_i).$
There exists a critical value of the parameter $s,$ denoted by
$h^P(Z,\epsilon),$ where $\EuScript{P}^s_{\epsilon}(Z)$ jumps from
$\infty$ to 0, i.e.,

\begin{linenomath*}
 \begin{eqnarray*}
  \EuScript{P}_\epsilon^s(Z)=
   \left\{\begin{aligned}
0, ~~~&s>h^P(Z,\epsilon),\\
\infty,  ~~~&s<h^P(Z,\epsilon).
 \end{aligned}\right.
  \end{eqnarray*}
   \end{linenomath*}
Since $h^P(Z,\epsilon)$ increases when $\epsilon$ decreases, we call

\begin{linenomath*}
 \begin{eqnarray*}
  h^P(Z):=\lim\limits_{\epsilon\to0}h^P(Z,\epsilon)
   \end{eqnarray*}
    \end{linenomath*}
the {\it packing topological entropy} of $Z.$

\noindent Remark that in the definition of
$\EuScript{P}_\epsilon^s(Z), \bigcup\limits_{i=1}^\infty Z_i\supset
Z$ can be replaced by $\bigcup\limits_{i=1}^\infty Z_i= Z.$

\subsection{Upper capacity topological entropy} Upper capacity
topological entropy is the straightforward generalization of the
Adler-Konheim-McAndrew definition of topological entropy to
arbitrary subsets.

Given a non-empty subset $Z\subset X.$ For $\epsilon>0,$ a set
$E\subset Z$ is called a $(n,\epsilon)$-{\it separated set} of $Z$
if $x,y\in E,x\neq y$ implies that $d_n(x,y)>\epsilon; E\subset X$
is called $(n,\epsilon)$-{\it spanning set} of $Z,$ if for any $x\in
Z,$ there exists $y\in E$ with $d_n(x,y)\leq \epsilon.$ Let
$r_n(Z,\epsilon)$ denote the largest cardinality of
$(n,\epsilon)$-separated sets for $Z,$ and
$\widetilde{r}_n(Z,\epsilon)$ the smallest cardinality of
$(n,\epsilon)$-spanning sets of $Z.$ The {\it upper capacity
topological entropy }of $Z$ is given by

\begin{linenomath*}
 \begin{eqnarray*}
  h^U(Z)=\lim\limits_{\epsilon\to0}\limsup\limits_{n\to\infty}\frac{1}{n}\log
   r_n(Z,\epsilon)=\lim\limits_{\epsilon\to0}\limsup\limits_{n\to\infty}\frac{1}{n}\log\widetilde{r}_n(Z,\epsilon).
   \end{eqnarray*}
    \end{linenomath*}
Some properties of topological entropies are presented as below.

\begin{prop}\label{prop}
 \begin{description}
  \item[(i)] For $Z\subset Z', h^B(Z)\leq h^B(Z'),h^P(Z)\leq h^P(Z'),h^U(Z)\leq h^U(Z').$
  \item[(ii)] For $Z\subset\bigcup\limits_{i=1}^\infty Z_i,s\geq0$ and
  $\epsilon>0,$ we have

\begin{linenomath*}
 \begin{eqnarray*}
  \EuScript{M}_\epsilon^s(Z)&\leq&\sum\limits_{i=1}^\infty
   \EuScript{M}_\epsilon^s(Z_i),\\
h^B(Z)&\leq&\sup\limits_{i\geq1}h^B(Z_i),\\
h^P(Z)&\leq&\sup\limits_{i\geq1}h^P(Z_i).
\end{eqnarray*}
 \end{linenomath*}
  \item[(iii)] For any $Z\subset X, h^B(Z)\leq h^P(Z)\leq h^U(Z).$
   \item[(iv)] If $Z$ is $T$-invariant and compact, then $h^B(Z)=h^P(Z)=h^U(Z).$
    \item[(v)] For $Z_1,Z_2\subset X,$ we have $h^U(Z_1\times Z_2)\leq h^U(Z_1)+h^U(Z_2).$
     \item[(vi)] For $Z_1,Z_2\subset X,$ we have $h^U(Z_1\cup Z_2)=\max\left\{h^U(Z_1),h^U(Z_2)\right\}.$
\end{description}
 \end{prop}

\begin{proof}
(i)-(iv) can be seen in \cite{FenHua}. To see (v), for
$n\in\mathbb{N},\epsilon>0,$ suppose $S_1$ is an
$(n,\epsilon)$-spanning sets of $Z_1$ with minimal cardinality and
$S_2$ is an $(n,\epsilon)$-spanning sets of $Z_2$ with minimal
cardinality, then $S_1\times S_2$ is an $(n,\epsilon)$-spanning set
of $Z_1\times Z_2.$ This means $\widetilde{r}_n(Z_1\times
Z_2,\epsilon)\leq
\widetilde{r}_n(Z_1,\epsilon)\widetilde{r}_n(Z_2,\epsilon).$
Furthermore,

\begin{linenomath*}
 \begin{eqnarray*}
  h^U(Z_1\times
  Z_2)&=&\lim\limits_{\epsilon\to0}\limsup\limits_{n\to\infty}\frac{1}{n}\log\widetilde{r}_n(Z_1\times
  Z_2,\epsilon)\\
  &\leq&\lim\limits_{\epsilon\to0}\limsup\limits_{n\to\infty}\frac{1}{n}\log\widetilde{r}_n(Z_1,\epsilon)\widetilde{r}_n(Z_2,\epsilon)\\
   &\leq&\lim\limits_{\epsilon\to0}\limsup\limits_{n\to\infty}\frac{1}{n}\log\widetilde{r}_n(Z_1,\epsilon)+\lim\limits_{\epsilon\to0}\limsup\limits_{n\to\infty}\frac{1}{n}\log\widetilde{r}_n(Z_2,\epsilon)
    \\ &=&h^U(Z_1)+h^U(Z_2).
     \end{eqnarray*}
      \end{linenomath*}
(vi) follows from that for $n\in\mathbb{N},\epsilon>0,$

\begin{linenomath*}
 \begin{eqnarray*}
  \max\{r_n(Z_1,\epsilon), r_n(Z_2,\epsilon))\} &\leq &r_n(Z_1\cup
Z_2,\epsilon)\\ &\leq& r_n(Z_1,\epsilon)+r_n(Z_2,\epsilon)\\ &\leq&
2\max\left\{r_n(Z_1,\epsilon),r_n(Z_2,\epsilon)\right\}.
    \end{eqnarray*}
     \end{linenomath*}
      \end{proof}

The main result of this article is the following theorem.

\begin{thm}
Suppose $(X,d,T)$ is a topological dynamical system.
\begin{linenomath*}
 \begin{itemize}
  \item If $h^B(X)<\infty$ and $Z_1\subset X, Z_2\subset X$ are analytic, then
   \begin{equation*}
h^B(Z_1)+h^{B}(Z_2)\leq h^B(Z_1\times Z_2);
\end{equation*}
 \item If $Z_1\subset X, Z_2\subset X,$ then
  \begin{eqnarray*}
   h^B(Z_1\times Z_2)&\leq& h^B(Z_1)+h^{P}(Z_2);\\
    h^P(Z_1\times Z_2)&\leq& h^P(Z_1)+h^{P}(Z_2);\\
     h^U(Z_1\times Z_2)&\leq& h^U(Z_1)+h^{U}(Z_2).\\
\end{eqnarray*}
 \end{itemize}
  \end{linenomath*}
   \end{thm}
\section{Proof of Main Theorem}
The main theorem is divided into several theorems, which are proved
respectively.
 The following lemma establishes the variational principles for Bowen and packing topological entropies
 of arbitrary Borel sets.

\begin{lem}{\rm \cite{FenHua}}\label{lemma}
Suppose $(X,d,T)$ is a topological dynamical system.

\begin{description}
 \item[(i)] If $K\subset X$ is non-empty and compact, then

  \begin{linenomath*}
   \begin{eqnarray*}
    h^B(K)=\sup\{\underline{h}_\mu(T):\mu\in M(X),\mu(K)=1\}.
     \end{eqnarray*}
      \end{linenomath*}
 \item[(ii)] Assume that $h^B(X)<\infty.$ If $Z\subset X$ is
  analytic, then

   \begin{linenomath*}
    \begin{eqnarray*}
     h^B(Z)=\sup\{h^B(K): K\subset Z {\rm~is~compact}\}.
      \end{eqnarray*}
       \end{linenomath*}
        \item[(iii)] If $K\subset X$ is non-empty and compact, then

\begin{linenomath*}
 \begin{eqnarray*}
  h^P(K)=\sup\{\overline{h}_\mu(T):\mu\in M(X),\mu(K)=1\}.
   \end{eqnarray*}
    \end{linenomath*}

  \item[(iv)] If $Z\subset X$ is
   analytic, then

   \begin{linenomath*}
      \begin{eqnarray*}
       h^P(Z)=\sup\{h^P(K): K\subset Z
    {\rm~is~compact}\}.
 \end{eqnarray*}
  \end{linenomath*}
   \end{description}
    \end{lem}
It is worth pointing out that the
 product metric in this article,  denoted by $\rho,$ on $(X\times X)$ is given by

\begin{linenomath*}
 \begin{eqnarray*}
  \rho((x_1,y_1),(x_2,y_2))=\max\{d(x_1,x_2),d(y_1,y_2)\}
   \end{eqnarray*}
    \end{linenomath*}
for $x_1,y_1,x_2,y_2\in X.$

Now, we firstly prove the following theorem.

\begin{thm}\label{theorem}
Suppose $(X,d,T)$ is a topological dynamical system.

\begin{itemize}
 \item If $h^B(X)<\infty$ and $Z_1\subset X, Z_2\subset X$ are analytic, then

\begin{linenomath*}
 \begin{equation*}
  h^B(Z_1)+h^{B}(Z_2)\leq h^B(Z_1\times Z_2).
   \end{equation*}
    \end{linenomath*}

  \item If $Z_1\subset X, Z_2\subset X,$ then $ h^B(Z_1\times Z_2)\leq
  h^B(Z_1)+h^{U}(Z_2).$
\end{itemize}

\end{thm}

\begin{proof}
(i) Firstly, we show that $h^B(Z_1)+h^{B}(Z_2)\leq h^B(Z_1\times
Z_2).$ It follows from Lemma \ref{lemma} that for any $\zeta>0,$
there exist $K_1\subset Z_1,\mu_1\in M(X), K_2\subset Z_2,\mu_2\in
M(X)$ such that

\begin{itemize}
 \item $K_1$ and $K_2$ are compact.
  \item $\mu_1(K_1)=1$ and $\mu_2(K_2)=1.$
   \item $h^B(Z_1)\leq \underline{h}_{\mu_1}(T)+\zeta/2$ and $h^B(Z_2)\leq \underline{h}_{\mu_2}(T)+\zeta/2.$
    \end{itemize}
Then $K_1\times K_2\subset Z_1\times Z_2$ is compact,
$\mu_1\times\mu_2\in M(X\times X),\mu_1\times\mu_2(K_1\times K_2)=1$
and

\begin{linenomath*}
 \begin{eqnarray*}
h^B(Z_1\times Z_2) &\geq & \underline{h}_{\mu_1\times \mu_2}(T\times
T)\\&=&\int
 \underline{h}_{\mu_1\times\mu_2}(T\times
 T,(x,y))d\mu_1\times\mu_2(x,y)\\
 &=&\int\lim\limits_{\epsilon\to0}\liminf\limits_{n\to\infty}-\frac{1}{n}\log\mu_1\times\mu_2
 B_n((x,y),\epsilon) d\mu_1\times\mu_2(x,y)\\
 &=&\int\lim\limits_{\epsilon\to0}\liminf\limits_{n\to\infty}-\frac{1}{n}\log\mu_1\times\mu_2
 B_n(x,\epsilon)\times  B_n(y,\epsilon) d\mu_1\times\mu_2(x,y)\\
 &\geq&\int\lim\limits_{\epsilon\to0}\liminf\limits_{n\to\infty}-\frac{1}{n}\log\mu_1
 B_n(x,\epsilon)+ \lim\limits_{\epsilon\to0}\liminf\limits_{n\to\infty}-\frac{1}{n}\log\mu_2 B_n(y,\epsilon)
 d\mu_1\times\mu_2(x,y)\\&=&\underline{h}_{\mu_1}(T)+\underline{h}_{\mu_2}(T)\\
 &\geq& h^B(Z_1)+h^B(Z_2)-\zeta.
\end{eqnarray*}
 \end{linenomath*}
Letting $\zeta\to0,$ we get the desired result.

Secondly, we prove that $h^B(Z_1\times Z_2)\leq
h^B(Z_1)+h^{U}(Z_2).$

Let $Z_2\subset X$ and assume $s>h^U(Z_2).$ For any $n\in\mathbb{N}$
and $\epsilon>0,$ let $R=R_n(Z_2,\epsilon)$ be the largest number so
that there is a disjoint family
$\{\overline{B}_n(x_i,\epsilon)\}_{i=1}^R$ with $x_i\in Z_2.$ Then
it is easy to see that for any $\delta>0,$

\begin{linenomath*}
 \begin{eqnarray*}
  \bigcup\limits_{i=1}^R B_n(x_i,2\epsilon+\delta)\supset
Z_2,
   \end{eqnarray*}
    \end{linenomath*}
which implies that $\EuScript{M}_{n,2\epsilon+\delta}^s(Z_2)\leq
R\exp(-ns)\leq 1.$ Let $t>h^B(Z_1),$ then $0=\EuScript{M}^t(Z_1)\geq
\EuScript{M}_\epsilon^t(Z_1) \geq 0.$ Hence, for any $\xi>0,$ there
is a finite or countable family $\{B_{n_i}(x_i,\epsilon)\}$ such
that $x_i\in X, n_i\geq N, \bigcup_i B_{n_i}(x_i,\epsilon)\supset
Z_1$ and $\sum\limits_i\exp(-tn_i)<\xi.$ For each $i,$ we can cover
$Z_2$ with $R_{n_i}(Z_2,\epsilon)$ balls $B_{i,j},$ of order $n_i$
and radius $2\epsilon+\delta, j=1,2,\cdots,R_{n_i}(Z_2,\epsilon).$
Then the sets $B_{n_i}(x_i,2\epsilon+\delta)\times B_{i,j}$ together
cover $Z_1\times Z_2.$ We have

\begin{linenomath*}
 \begin{eqnarray*}
  \EuScript{M}^{s+t}_{N,2\epsilon+\delta}(Z_1\times
Z_2)&\leq&\sum\limits_{i}\exp(-n_i(s+t))R_{n_i}(Z_2,\epsilon)\\
&=&\sum\limits_{i}\exp(-n_is)\exp(-n_it)R_{n_i}(Z_2,\epsilon)\\
&\leq&\xi.
\end{eqnarray*}
 \end{linenomath*}
This implies that $h^B(Z_1\times Z_2)\leq s+t.$ Letting $t\to
h^B(Z_2),s\to h^U(Z_1),$ this completes the proof.
\end{proof}
A question arises naturally whether $h^U$ can be replaced by $h^P$
in Theorem \ref{theorem}. For this purpose, we present an equivalent
definition of packing topological entropy in the following
proposition.

\begin{prop}\label{prop2}
$h^P(Z)=\inf\left\{\sup\limits_i h^U(Z_i): Z=\bigcup\limits_i Z_i
\right\}.$
 \end{prop}

\begin{proof}
 Given
$a>h^P(Z)=\lim\limits_{\epsilon\to0}h^P(Z,\epsilon)=\lim\limits_{\epsilon\to0}\inf\left\{s:\EuScript{P}_\epsilon^s(Z)=0\right\}.$
Since $h^P(Z,\epsilon)$ increases as $\epsilon$ decreases, we have
$\EuScript{P}_\epsilon^a(Z)=0.$ Furthermore, there exists
$\{Z_i\}_{i=1}^\infty$ such that $\bigcup\limits_{i=1}^\infty Z_i=
Z$ and $\sum\limits_{i=1}^\infty P_\epsilon^a(Z_i)<1.$ This implies
that $P_\epsilon^a(Z_i)<1,$ for  such $\{Z_i\}_{i=1}^\infty.$ Since
for $N\in\mathbb{N},\epsilon>0, s>0$ and any subset $B\subset X,$ we
have $r_N(B,2\epsilon)\exp(-Ns)\leq P^s_{N,\epsilon}(B).$ Then for
any $Z_i\in\{Z_i\}_{i=1}^\infty,$

\begin{linenomath*}
 \begin{eqnarray*}
  \lim\limits_{\epsilon\to0}
   \limsup\limits_{N\to\infty}r_N(Z_i,2\epsilon)\exp(-Na) &\leq&
    \lim\limits_{\epsilon\to0} \limsup\limits_{N\to\infty}
P_{N,\epsilon}^a(Z_i)\\&=&\lim\limits_{\epsilon\to0}P^a_\epsilon(Z_i)\leq1.
\end{eqnarray*}
 \end{linenomath*}
This implies that
$\lim\limits_{\epsilon\to0}\limsup\limits_{N\to\infty}\frac{\log
r_N(Z_i,2\epsilon)}{N}\leq a,$ i.e., $h^U(Z_i)\leq a.$ Furthermore,

\begin{linenomath*}
 \begin{eqnarray*}
  \inf\left\{\sup\limits_{i\geq1}h^U(Z_i):\bigcup\limits_{i=1}^\infty
   Z_i=Z\right\}\leq\sup\limits_{i\geq1} h^U(Z_i)\leq a.
    \end{eqnarray*}
     \end{linenomath*}
Letting $a\to h^P(Z),$ we have

\begin{linenomath*}
 \begin{eqnarray*}
  \inf\left\{\sup\limits_{i\geq1}h^U(Z_i):\bigcup\limits_{i=1}^\infty
   Z_i=Z\right\}\leq h^P(Z).
    \end{eqnarray*}
     \end{linenomath*}

To prove the opposite inequality, let $0<t<s<h^P(Z).$ Given
$\{Z_i\}_{i=1}^\infty$ such that $\bigcup\limits_{i=1}^\infty
Z_i=Z,$ it is enough to show that $h^U(Z_i)\geq t$ for some $i.$
Since $h^P(Z)=\lim\limits_{\epsilon\to0}h^P(Z,\epsilon),$ there
exists $\epsilon^*>0$ such that for any $0<\epsilon<\epsilon^*,
h^P(Z,\epsilon)>s.$ Then $\inf\left\{\sum\limits_{i=1}^\infty
P_\epsilon^s(Z_i):\bigcup\limits_{i=1}^\infty
Z_i=Z\right\}=\EuScript{P}_\epsilon^s(Z)>0.$ Furthermore, there
exist $\alpha$ and $i,$ such that $P_\epsilon^s(Z_i)>\alpha>0.$
Since $P_{N,\epsilon}^s(Z_i)$ decreases as $N$ increases, we have $
P_{N,\epsilon}^s(Z_i)>\alpha.$ There exist $N_1$ and a finite or
countable pairwise disjoint family
$\{\overline{B}_{n_j}(x_j,\epsilon)\}_{j=1}^\infty$ such that
$x_j\in Z_i,n_j\geq N_1$ for all $j$ and
$\sum\limits_{j=1}^\infty\exp(-sn_j)>\alpha.$ For each
$k\in\mathbb{N},$ let $m_k$ be the number of $j$ such that $n_j=k.$
Then we have $\sum\limits_{k=N_1}^\infty m_k\exp(-ks)>\alpha.$ This
yields for some $N\geq N_1, m_N\geq \exp(Nt)(1-\exp(t-s))\alpha,$
since otherwise

\begin{linenomath*}
 \begin{equation*}
  \sum\limits_{j=1}^\infty\exp(-sn_j)=\sum\limits_{k=N_1}^\infty
   m_k\exp(-ks)<\sum\limits_{k=0}^\infty
    \exp(k(t-s))(1-\exp(t-s))\alpha=\alpha.
     \end{equation*}
      \end{linenomath*}
Furthermore,

\begin{linenomath*}
 \begin{eqnarray*}
  h^U(Z_i)=\lim\limits_{\epsilon\to0}\limsup\limits_{N\to\infty}\frac{1}{N}\log
  r_N(Z_i,\epsilon)\geq \limsup\limits_{N\to\infty}\frac{1}{N}\log
  m_N\geq t.
   \end{eqnarray*}
    \end{linenomath*}
This completes the proof.
\end{proof}

\begin{thm}
Suppose $(X,d,T)$ is a topological dynamical system. If $Z_1\subset
X, Z_2\subset X$, then

\begin{linenomath*}
 \begin{equation*}
  h^B(Z_1\times Z_2)\leq h^B(Z_1)+h^{P}(Z_2).
   \end{equation*}
    \end{linenomath*}
\end{thm}

\begin{proof}
Given $\{Z_2^i\}_{i=1}^\infty$ such that
$\bigcup\limits_{i=1}^\infty Z_2^i=Z_2.$

\begin{linenomath*}
 \begin{eqnarray*}
  h^B(Z_1\times Z_2)
  &=& h^B(Z_1\times\bigcup\limits_{i=1}^\infty Z_2^i)\\
  &=& h^B(\bigcup\limits_{i=1}^\infty(Z_1\times Z_2^i))\\
  &=& \sup\limits_{i\geq1}h^B(Z_1\times Z_2^i)\\
  &\leq& h^B(Z_1)+\sup\limits_{i\geq1}h^U(Z_i).
    \end{eqnarray*}
     \end{linenomath*}
This together with Proposition \ref{prop2} implies that
$h^B(Z_1\times Z_2)\leq h^B(Z_1)+h^P(Z_2).$
\end{proof}

\begin{thm}
Suppose $(X,d,T)$ is a topological dynamical system. If $Z_1\subset
X, Z_2\subset X$, then

\begin{linenomath*}
 \begin{equation*}
  h^P(Z_1\times Z_2)\leq h^P(Z_1)+h^{P}(Z_2).
   \end{equation*}
    \end{linenomath*}
     \end{thm}

\begin{proof}
This follows from Proposition \ref{prop2} and (v) of Proposition
\ref{prop}.
\end{proof}

\begin{cro}
For any $A\subset X,$ and $\epsilon>0$ there exists an increasing
sequence $A_1\subset A_2\subset\cdots\subset A$   such that
$A=\bigcup\limits_{i=1}^\infty A_i$ and $h^U(A_i)\leq
h^P(A)+\epsilon.$
\end{cro}

\begin{proof}
By Proposition \ref{prop2}, for $\epsilon,$ there exists $A_1$ such
that $h^U(A_1)<h^P(A)+\epsilon.$ Then for $\epsilon/2,$ there exists
$A'_2$ such that $h^U(A'_2)<h^P(A)+\epsilon/2.$ Let $A_2=A_1\cup
A'_2.$ Then $h^U(A_2)=\max\{h^U(A_1),h^U(A'_2)\}<h^P(A)+\epsilon.$
We can construct a sequence $\{A_i\}_{i=1}^\infty$ like this. Such
$\{A_i\}_{i=1}^\infty$ is desired.
\end{proof}
It is worth pointing out that the above results hold for two
different topological dynamical systems, i.e., suppose
$(X_1,d_1,T_1),(X_2,d_2,T_2)$ are two topological dynamical systems
and the product metric  $\rho$ on $(X_1\times X_2,T_1\times T_2)$ is
given by
$\rho((x_1,y_1),(x_2,y_2))=\max\{d_1(x_1,x_2),d_2(y_1,y_2)\}$ for
any $x_1,y_1\in X_1,x_2,y_2\in X_2.$

\begin{itemize}
 \item If  $h^B(X_1)<\infty$ and $h^B(X_2)<\infty,$ and $Z_1\subset
X_1,Z_2\subset X_2$ are analytic, then

\begin{linenomath*}
 \begin{eqnarray*}
  &h^B(Z_1)+h^B(Z_2) \leq h^B(Z_1\times Z_2);
   \end{eqnarray*}
    \end{linenomath*}
     \item If $Z_1\subset X_1,Z_2\subset X_2,$ then

\begin{linenomath*}
 \begin{eqnarray*}
  h^B(Z_1\times Z_2) &\leq&
  h^B(Z_1)+h^P(Z_2);\\
  h^P(Z_1\times Z_2)&\leq& h^P(Z_1)+h^P(Z_2);\\
  h^U(Z_1\times Z_2)&\leq& h^U(Z_1)+h^U(Z_2).
  \end{eqnarray*}
   \end{linenomath*}
    \end{itemize}
This together with (iv) of Proposition \ref{prop} leads to the
following corollary.

\begin{cro}
Suppose $(X_1,d_1,T_1)$ and $(X_2,d_2,T_2)$ are two topological
dynamical systems. If $Z_1\subset X_1, Z_2\subset X_2$ are analytic
and $Z_2$ is $T_2$-invariant and compact (or $Z_1$ is
$T_1$-invariant and compact), then $h^B(Z_1\times
Z_2)=h^B(Z_1)+h^B(Z_2).$
\end{cro}
At last, we give an example as follows.

\textbf{Example:} We take a topological dynamical system $(X,T)$ and
$D=\left\{\frac{1}{n}\right\}_{n\in\mathbb{N}}\cup\{0\}$ and let
$Z=X\times D.$ Define $R:Z\to Z$ satisfying
$R(x,\frac{1}{n+1})=(x,\frac{1}{n}), n\in\mathbb{N}; R(x,1)=(Tx,1)$
and $R(x,0)=(x,0)$ for $x\in X.$ Then $(Z,R)$ is a topological
dynamical system. If we identity $(x,1)$ with $x$ for each $x\in X,$
then $X$ can be viewed as a closed subset of $Z$ and $R|_X=T.$ Since
$h^P(D)\leq\max \left\{
\sup\limits_{n\in\mathbb{N}}h^{U}(\{\frac{1}{n}\}),h^U(\{0\})\right\}=0,$
we have $h^B(Z)=h^B(X).$


\noindent {\bf Acknowledgements.}  The research was supported by the
National Natural Science Foundation of China (Grant No. 11271191)
and National Basic Research Program of China (Grant No.
2013CB834100) and the Foundation for Innovative Program of Jiangsu
Province (Grant No. CXZZ12 0380).

\end{document}